\documentclass[11pt,reqno]{amsart}
\usepackage{amsmath}
\usepackage{amsthm}
\usepackage{amssymb}
\usepackage{graphicx}
\usepackage{amsfonts}
\usepackage{latexsym}
\usepackage{setspace}
\usepackage[backref=page]{hyperref}
\usepackage{cite}
\usepackage{multicol}
\usepackage[active]{srcltx}
\usepackage{mathrsfs}
\usepackage{yhmath}
\usepackage{color}
\usepackage{mathdots}

\usepackage{caption}
\usepackage{subcaption}

\newcommand{\C}{ \mathbb{C}}
\newcommand{\D}{ \mathbb{D}}
\newcommand{\T}{\mathbb{T}}
\newcommand{\Z}{ \mathbb{Z}}

\newcommand{\R}{ \mathbb{R}}

\newcommand{\dist}{\operatorname{dist}}
\newcommand{\norm}[1]{\left\| #1 \right\|}
\newcommand{\inner}[1]{\left< #1 \right>}

\newcommand{\N}{\mathbb{N}}

\renewcommand{\Re}[1]{\operatorname{Re}#1}
\renewcommand{\Im}[1]{\operatorname{Im}#1}

\renewcommand{\phi}{\varphi}

\numberwithin{equation}{section}
\theoremstyle{plain}
\newtheorem{Proposition}[equation]{Proposition}
\newtheorem{Corollary}[equation]{Corollary}
\newtheorem{Theorem}[equation]{Theorem}
\newtheorem{Lemma}[equation]{Lemma}

\theoremstyle{definition}

\newtheorem{Example}[equation]{Example}
\newtheorem{Remark}[equation]{Remark}
\newtheorem{Question}[equation]{Question}

\allowdisplaybreaks

\def\HHH{\mathfrak{H}}
\newcommand{\ind}{\mathop{\rm ind}}
\def\bbD{\mathbb{D}}
\newcommand{\Onr}{\OOO_{n,r}}
\def\OOO{\mathfrak{O}}
\newcommand{\Arg}{\mathop{\rm Arg}}
\newcommand{\Arc}[2]{\lfloor#1,#2\rfloor}
\def\bbR{\mathbb{R}}

\def\bbH{\mathbb{H}}
\newcommand{\sS}{\mathcal{S}}

\begin{document}

\bibliographystyle{amsplain}

    \title{An extremal problem for characteristic functions}

\author[I.~Chalendar]{Isabelle Chalendar}
\address{Institut Camille Jordan, University of Lyon I, 43 Boulevard du 11 Bovembre 1918,
69622 Villeurbanne cedex,
France}
\email{chalendar@math.univ-lyon1.fr}
\urladdr{http://math.univ-lyon1.fr/~chalenda/calendar.html}

    \author[S.R.~Garcia]{Stephan Ramon Garcia}
     \address{   Department of Mathematics\\
            Pomona College\\
            Claremont, California\\
            91711 \\ USA}
    \email{Stephan.Garcia@pomona.edu}
    \urladdr{http://pages.pomona.edu/~sg064747}

    \author[W.T.~Ross]{William T. Ross}
    \address{Department of Mathematics and Computer Science\\
            University of Richmond\\
            Richmond, Virginia\\
            23173 \\ USA}
    \email{wross@richmond.edu}
    \urladdr{http://facultystaff.richmond.edu/~wross}

	\author[D.~Timotin]{Dan Timotin}
	\address{Simion Stoilow Institute of Mathematics of the Romanian Academy, PO Box 1-764, Bucharest 014700, Romania}
	\email{Dan.Timotin@imar.ro}
	\urladdr{http://www.imar.ro/~dtimotin}


    \thanks{Second author partially supported by National Science Foundation Grant DMS-1001614. Fourth author partially supported by a grant of the Romanian National Authority for Scientific
Research, CNCS Ð UEFISCDI, project number PN-II-ID-PCE-2011-3-0119.}

\begin{abstract}
	Suppose $E$ is a subset of the unit circle $\T$ and $H^\infty\subset L^\infty$ is the Hardy subalgebra. We examine the problem of finding the distance from the characteristic function of $E$ to $z^nH^\infty$. This admits an alternate description as a dual extremal problem. Precise solutions are given in several important cases. The techniques used involve the theory of Toeplitz and Hankel operators as well as the construction of certain conformal mappings.
\end{abstract}

\maketitle

\section{Introduction}

	The linear extremal problem 
	\begin{equation}\label{eq-ExtremalGeneral}
		\Lambda(\psi) := \sup_{F \in b(H^1)} \left|\frac{1}{2 \pi i} \int_{\T} \psi(\zeta) F(\zeta) \,d\zeta \right|,
	\end{equation}
	where $b(H^1)$ is the unit ball of the classical Hardy space $H^1$ \cite{Duren, Ga} on the open unit disk $\D$ and
	$\psi$ is in $L^{\infty}$ of the unit circle $\T$, has been studied by many different authors
	over the last century (see \cite{NLEPHS} for a brief survey and a list of references). 
For some historical context, let us mention an early result  due to Fej\'{e}r \cite{Fejer},
	which says that for any complex numbers $c_0, c_1,\ldots, c_n$ one has
	\begin{equation*}
		\Lambda\left(\frac{c_0}{z} + \cdots + \frac{c_{n}}{z^{n + 1}}\right) = \norm{H},
	\end{equation*}
	where $H$ is the Hankel matrix
	(blank entries to be treated as zeros)
	\begin{equation*}
		H =
		\begin{pmatrix}
			 c_0 & c_1 & c_2 & \cdots & c_n \\
			 c_1 & c_2 & \cdots & c_n &  \\
			 c_2 & \cdots & c_n & &\\
			 \vdots & \iddots & &  &  \\
			 c_n &  &  & &
		\end{pmatrix}
	\end{equation*}
	and $\norm{H}$ denotes the operator norm of $H$ (i.e., the largest singular value of $H$). In the special case when $c_j = 1$, $0 \leq j \leq n$, Egerv\'{a}ry \cite{Egervary} obtained explicit formulae for $\Lambda$ and for the extremal function $F$.
	Fej\'{e}r's result was generalized by Nehari in terms of Hankel operators (see \cite{Nehari}, \cite[Theorem 1.1.1]{Peller} and \eqref{norm-Hankel-dist} below). 
	We refer the reader to \cite{Duren, Ga, Kha-Sr} for further references.

	It is also known \cite{CFT, NLEPHS}
	that for each $\psi \in L^{\infty}$ the supremum in \eqref{eq-ExtremalGeneral} is equal to
	\begin{equation}\label{eq-ExtremalQuadratic}
		\sup_{f \in b(H^2)} \left|\frac{1}{2 \pi i} \int_{\T} \psi(\zeta) f(\zeta)^2 \,d\zeta \right|,
	\end{equation}
	which is a quadratic extremal problem posed over the unit ball of the Hardy space $H^2$.  For rational $\psi$ there are techniques in \cite{CFT, NLEPHS} which lead not only to the supremum in \eqref{eq-ExtremalGeneral} but also to the extremal function $F$ for which the supremum is attained. For general $\psi \in L^{\infty}$ the supremum is difficult to compute and the extremal function may not exist or, even when it does exist, it may not be unique. 
		
	In this paper we discuss the family of extremal problems corresponding to 
	\begin{equation}\label{eq-Symbol}
		\psi(z) = \chi_E(z) \overline{z}^{n},
	\end{equation}
	where $E$ is a Lebesgue measurable subset of $\T$, $\chi_E$ denotes the characteristic function of $E$, and $n \in \Z$. We are thus interested in the quantities
	\begin{equation}\label{eq-Lambda}
		\Lambda_{n}(E) := \sup_{F \in b(H^1)} 
		\left|\frac{1}{2 \pi i} \int_E  F(\zeta) \overline{\zeta}^{n} \, d\zeta\right|
	\end{equation}
	
		If we note that
	\begin{equation*}
		\Lambda_n(\T) = \sup_{F \in b(H^1)} \big| \widehat{F}(n-1)\big|,
	\end{equation*}
	one may  interpret \eqref{eq-Lambda} as asking \emph{how large can be the contribution of the set $E$ to the $(n-1)$st Fourier coefficient of an $H^1$ function?}.
	
The main results of this paper are as follows. Using Hankel operators and distribution estimates for harmonic conjugates, we first show in Section~\ref{Section:Negative} that if $E$ has Lebesgue measure $|E| \in (0, 2 \pi)$ then
			$$\Lambda_{n}(E) = \tfrac{1}{2}, \quad n \leq 0.$$
			
For $n\ge 1$, we obtain some general estimates in Section~\ref{Section:Positive}. The central part of the paper is contained in Section~\ref{se:interval}, where we use a conformal mapping argument, along with our Hankel operator techniques, to obtain a formula for $\Lambda_{n}(I)$ when $n \geq 1$ and $I$ is an arc of the circle $\T$. It is remarkable that in the case $n=1$ the formula is explicit and may be extended to any measurable set $E\subset \T$ with $|E| \in (0, 2 \pi]$ (as shown in Section~\ref{se:more general sets}); namely, we have 
$$\Lambda_{1}(E) = \frac{1}{2} \sec\left(\frac{\pi \frac{|E|}{2}}{\pi + |E|}\right).$$			
			
	For $n\ge 2$ and arbitrary measurable sets $E$, one obtains  an upper bound for $\Lambda_n(E)$.

	Our work stems, perhaps somewhat surprisingly, from additive number theory.
	As noted in \cite[p.~325]{MTB}, an approach to the Goldbach conjecture
	using the Hardy-Littlewood circle method requires a bound on the expressions
	\begin{equation}\label{eq-Goldbach}
		\left| \int_{\mathfrak{m}} f(x)^2 e(-nx)\,dx \right|
	\end{equation}
	where $e(x) = \exp(2\pi i x)$, $f$ is a certain polynomial in $e(x)$, and $\mathfrak{m}$
	is a particular disjoint union of sub-intervals of $[0,1]$.  With slight adjustments in notation, one
	observes that problems of the form \eqref{eq-ExtremalQuadratic}, where $\psi$ is given by \eqref{eq-Symbol},
	are of some relevance to estimating the quantity in \eqref{eq-Goldbach}.  A
	similar approach, where the exponent $2$ replaced by $3$ in \eqref{eq-Goldbach},
	was famously used by I. M.~Vinogradov in his celebrated proof 
	that every sufficiently large odd integer is the sum of three primes. 
Needless to say, we do not expect our results to help solve the Goldbach conjecture.  We merely point out how these 
	extremal problems are of interest outside  complex analysis and operator theory. 	
	
\section{Preliminaries}\label{Section:Preliminaries}

	For $p \in [1, \infty]$, we let $L^p$ represent the standard Lebesgue spaces on the circle (with respect to normalized Lebesgue measure), with the integral norms $\|\cdot\|_{p}$ for finite $p$ 
and with the essential supremum norm $\|\cdot\|_{\infty}$ for $L^{\infty}$. Let $\mathcal{C}:= \mathcal{C}(\mathbb{T})$ denote the complex valued continuous functions on $\T$ with the supremum norm $\|\cdot\|_{\infty}$. 	
	We let $H^p$ denote the classical Hardy spaces and $H^{\infty}$ denote the space of bounded analytic  functions on $\D$. As is standard, we identify $H^p$ with a closed subspace of $L^p$ via non-tangential boundary values on $\T$. See \cite{Duren, Ga} for a thorough treatment of this. 
	
	For $f \in L^1$ and $n \in \Z$, let 
	$$\widehat{f}(n) := \int_{-\pi}^{\pi} f(e^{i \theta}) e^{-i n \theta} \frac{d \theta}{2 \pi}$$
	denote $n$-th the Fourier coefficient of $f$ and let
	\begin{equation} \label{HC}
	\widetilde{f}(e^{i \theta}) := P.V. \int_{-\pi}^{\pi} f(e^{i(\theta - \varphi)}) \cot\left(\frac{\theta}{2}\right) \frac{d \varphi}{2 \pi}
	\end{equation}
	denote the harmonic conjugate of $f$. 

	In what follows, $E$ will be a Lebesgue measurable subset of the unit circle $\T$. We use $|E|$ to denote the (non-normalized) Lebesgue measure on $\T$
	so that $|E| \in [0, 2 \pi]$, so that $|I|$ coincides with arc length whenever $I$ is an arc on $\T$. We will use $E^{-}$ to denote the closure of $E$.
	
	Just so we are not dealing with trivialities in our extremal problem $\Lambda_{n}(E)$, we first dispose of the endpoint cases $|E| = 0$ and $|E| = 2 \pi$. Indeed 
	\begin{equation} \label{endpoint0}
	|E| = 0 \Rightarrow \Lambda_{n}(E) = 0, \quad n \in \Z,
	\end{equation}
	\begin{equation} \label{endpoint2pi}
	|E| = 2 \pi \Rightarrow \Lambda_{n}(E) = \begin{cases} 1 &\mbox{if } n \geq 1 \\ 
0& \mbox{if } n < 1.
	\end{cases}
	\end{equation}

	A natural first step in considering any linear extremal problem is to identify the corresponding dual extremal problem.  In this case (see \cite{Ga} for the details), we can use the tools of functional analysis to rephrase the original extremal problem in \eqref{eq-ExtremalGeneral} as the dual extremal problem
	\begin{equation}\label{eq:definition of lambda_n with distance}
		 \Lambda_n(E)= \dist(\bar z^n\chi_E, H^\infty) = \inf\{\|\overline{z}^{n} \chi_{E} - g\|_{\infty}: g \in H^{\infty}\}.
	\end{equation}
	It turns out that the above \emph{$\inf$} can be replaced by a \emph{$\min$} \cite[p.~146]{Koosis}. 
	Since $$\mbox{dist}(\overline{z}^n \chi_{E}, H^{\infty}) = \mbox{dist}(\chi_{E}, z^n H^{\infty})$$ we see, for a fixed set $E \subset \T$, that 
	\begin{equation} \label{increasinginn}
	\Lambda_n(E) \le \Lambda_{n'}(E), \quad n \leq n'.
	\end{equation}
	 A quick review of the definition of $\Lambda_{n}$  shows that it is invariant under rotation. In other words, 
	 \begin{equation} \label{rotation-invariant}
	 \Lambda_{n}(e^{i \theta} E) = \Lambda_{n}(E), \quad n \in \Z, \quad \theta \in [0, 2 \pi].
	 \end{equation}
	The key to our investigation is the fact that $\Lambda_{n}(E)$ can also be 
	expressed in terms of the norm of a certain Hankel operator.  
	
\subsection{Hankel operators}	
	The \emph{Hankel operator} with \emph{symbol} $\phi \in L^{\infty}$ is defined to be 
	$$\HHH_{\phi}:H^2 \to H^2_-, \quad \HHH_{\phi}f := P_-( \phi f),$$ where 
	$P_-$ denotes the orthogonal projection from $H^2$ onto $H^2_-:=L^2 \ominus H^2$.
	With respect to the orthonormal bases $\{1,z,z^2,\ldots\}$ for $H^2$ and $\{\bar z ,\bar z^2,\bar z^3,\ldots\}$
	for $H^2_-$,  $\HHH_{\phi}$ has the (Hankel) matrix representation
	\begin{equation*}
		\HHH_{\phi} = \big( \widehat{\phi}(-j-k-1) \big)_{0 \leq j, k < \infty}.
	\end{equation*}
	By Nehari's theorem (see \cite{Nehari} \cite[Theorem 1.1.1]{Peller}),
	\begin{equation} \label{norm-Hankel-dist}
	\norm{\HHH_{\phi}} = \dist(\phi,H^{\infty}),
	\end{equation}
	from which we conclude, via \eqref{eq:definition of lambda_n with distance}, that
	\begin{equation}\label{eq:definition of lambda_n with hankel}
		\Lambda_n(E)= \norm{ \HHH_{\bar z^n \chi_E} },
	\end{equation}
	where
	\begin{equation*}
		\HHH_{\bar z^n \chi_E}=\big(\widehat{\chi_E}(n-j-k-1)\big)_{0 \leq j, k < \infty}.
	\end{equation*}
	We  also make use of the formula 
	\begin{equation}\label{eq:EssentialDistance}
		\norm{ \HHH_{\phi}}_e = \dist(\phi,H^{\infty}+ \mathcal{C})
	\end{equation}
	for the essential norm of $\HHH_{\phi}$, that is, the distance to the compact operators \cite[1.5.3]{Peller}.  Here $H^{\infty}+\mathcal{C}$ denotes the uniformly closed
	algebra $\{ h + f : h \in H^{\infty}, f \in \mathcal{C}\}$ endowed with the supremum norm.  	

The  \emph{Toeplitz operator}  $T_{\varphi}$ with symbol $\phi \in L^{\infty}$ is defined to be
$$T_{\varphi}: H^2 \to H^2, \quad T_{\varphi} f = P_{+}(\varphi f),$$
where $P_{+}$ is the orthogonal projection of $L^2$ onto $H^2$. 	
	
	The following two results play a key role. The first one is an immediate consequence of~\cite[Thm.~7.5.5]{Peller} and its proof.
 
	\begin{Lemma}\label{Lemma:Unimodular}
		Let $u\in L^\infty$.
		
		\begin{enumerate}
		\item
		If $\|\HHH_u\|_e<\|\HHH_u\|$ and $\|u\|_\infty=\dist(u, H^\infty)$, then $u$ has constant absolute value almost everywhere.
		\item
		If  $u$ has constant absolute value almost everywhere, 
		the Toeplitz operator $T_u$ is Fredholm, and $\ind T_u>0$, then
		$$\|u\|_\infty=\dist(u, H^\infty).$$
		\end{enumerate}
	\end{Lemma}
	
The second one is from \cite{MR0259639} (see also \cite[Cor.~3.1.16]{Peller}). 

	\begin{Lemma}\label{Peller-2}
	Let $u$ be a unimodular function on $\mathbb{T}$. Then $T_u$ is Fredholm if and only if $\|\HHH_u\|_e<1$ and $\|\HHH_{\overline{u}}\|_e<1$.
	\end{Lemma}

\subsection{Two results on harmonic conjugates}
	In addition to the preceding results on Hankel operators, we also need a few classical results 
	from harmonic analysis.  Recall from \eqref{eq:definition of lambda_n with distance} that $\widetilde f$ is the harmonic conjugate of $f \in L^1$.
	Proofs of the following well-known result of Zygmund can be found in the standard texts
	\cite[V.D]{Koo} and \cite[Corollary III.2.6]{Ga}.
	
	\begin{Lemma}\label{le:zygmund}
		For each $\lambda<1$ there is a constant $C_\lambda>0$ such that if $f$ is real valued and 
		$\|f\|_\infty\leq \pi/2$ then  
		\begin{equation}\label{eq:zygmund1}
			\frac{1}{2\pi}\int_{-\pi}^\pi e^{\lambda |\widetilde f(e^{i \theta})|} d\theta \le C_\lambda.
		\end{equation}
		If $f$ is continuous, then for any  $\mu>0$ there exists a constant $C_{f,\mu}>0$ such that
		\begin{equation}\label{eq:zygmund2}
			\frac{1}{2\pi}\int_{-\pi}^\pi e^{\mu |\widetilde f(e^{i \theta})|} d\theta \le C_{f,\mu}.
		\end{equation}
	\end{Lemma}

	Given a real-valued function $f \in L^{\infty}$, we apply \eqref{eq:zygmund1}  and 
	Markov's inequality 
	$$\sigma(|g| \geq \epsilon) \leq \frac{1}{\epsilon} \int |g| \, d \sigma, \quad \epsilon > 0,$$
	for a positive measure $\sigma$, 
	to the function $$\frac{\pi f}{2 \norm{f}_{\infty}}$$ to obtain
	\begin{equation*}
		\left|\left\{\theta: |\widetilde f(e^{i \theta})|> \frac{2\|f\|_\infty}{\lambda\pi} \log t\right\}\right|
		\leq \frac{C_\lambda}{t}.
	\end{equation*}
	If $\norm{f}_\infty< \frac{1}{2}$, we may choose $\lambda<1$ such that $$\alpha:=\frac{2 \|f\|_\infty}{\lambda}<1.$$
	At this point, after some rewriting, we get
	\begin{equation}\label{eq:weak}
		\left|\{\theta: |\widetilde f(e^{i \theta})|>y\}\right|\leq C e^{-\frac{\pi y}{\alpha}},
	\end{equation}
	where $C$ is a constant which is independent of $y>0$. 
	Similarly, if $f$ is continuous, we obtain from~\eqref{eq:zygmund2} that for any $\mu>0$ there exists a $C_\mu > 0$ such that
	\begin{equation}\label{eq:weak2}
		\left|\{\theta: |\widetilde f(e^{i \theta})|>y\}\right|\leq C_\mu e^{-\mu y}.
	\end{equation}

	This next result of Stein and Weiss is an exact formula
	for the distribution function for $\widetilde \chi_E$ \cite{SW}. 
	
	\begin{Lemma}\label{le:stein-weiss}
		If $E$ is a Lebesgue measurable subset of $\T$ then
		\begin{equation*}
			\big|\{\theta: |\widetilde\chi_E(e^{i \theta})|>y\}\big|=4 \tan^{-1} \left(\frac{2\sin\frac{|E|}{2}}{e^{\pi y}-e^{-\pi y}}  \right).
		\end{equation*}
	\end{Lemma}

\subsection{An essential computation}

	Putting this all together, we are now ready to prove the following important lemma which is useful in our analysis of a certain family of Hankel operators.
	It is likely that this next result is already well-known as a `folk theorem', although we
	are unable to find a specific reference for it.
	
	\begin{Lemma}\label{Lemma:HC}
		If $E$ is a Lebesgue measurable subset of $\T$ and $|E| \in (0, 2 \pi)$, then
		\begin{equation*}
			\dist(\chi_E, H^\infty+\mathcal{C})= \tfrac{1}{2}.
		\end{equation*}
	\end{Lemma}
	
	\begin{proof}
		First observe that $\dist(\chi_E, H^\infty+\mathcal{C}) \leq \frac{1}{2}$ since $H^{\infty} + \mathcal{C}$ contains the constant function $\frac{1}{2}$.
		Suppose, toward a contradiction, that there exist functions $h$ in $H^\infty$ and $f \in \mathcal{C}$ 
		such that $\norm{ \chi_E-h-f}_\infty< \frac{1}{2}$.  Writing 
		\begin{equation*}
			h=u+i\widetilde u+ib,
		\end{equation*}
		where $u$ is real-valued and $b$ is a real constant, the inequality
		\begin{equation*}
			|\chi_E - u - \Re f|^2 +  | \widetilde u + b + \Im f|^2 =  |\chi_E-h - f|^2 < \tfrac{1}{4} 
		\end{equation*}
		holds almost everywhere $\T$ whence
		\begin{equation*}
			\norm{ \chi_E-u-  g }_\infty < \tfrac{1}{2},
		\end{equation*}
		where $g= \Re f$.  Letting $M = \norm{ \widetilde{u} }_{\infty}$, it follows that	
		\begin{equation}\label{eq:estimation two terms}
			\begin{split}
				\left|\{\theta: |\widetilde\chi_E(e^{i \theta})|>y\}\right|
				&\,\leq\, \left|\{\theta: |\widetilde\chi_E(e^{i \theta})-\widetilde u(e^{i \theta})-\widetilde g(e^{i \theta})|>\beta(y-M)\}\right|\\ 
				&\qquad+ \left|\{\theta: |\widetilde g(e^{i \theta})|>(1-\beta)(y-M)\}\right|
			\end{split}
		\end{equation}
		for any $0<\beta<1$.
		According to~\eqref{eq:weak}, the first term on the right hand side of \eqref{eq:estimation two terms} is majorized by 
		$$C e^{-\frac{\pi \beta (y - M)}{\alpha}}$$
		for some $\alpha<1$ and $C > 0$. Now choose $\beta<1$ such that $\beta/\alpha>1$, 
		and then $\mu>0$ such that $(1-\beta)\mu >\pi$. Applying~\eqref{eq:weak2}, it follows that the right 
		hand side of~\eqref{eq:estimation two terms} decreases in $y$ at least as fast as $\exp(-a\pi y)$ for some $a>1$, 
		while, by Lemma \ref{le:stein-weiss}, the order of decrease of the left hand side is exactly $\exp(-\pi y)$. 
		This leads to a contradiction if $y\to\infty$ which proves the lemma.
	\end{proof}

	\begin{Lemma}\label{Lemma:Essential}
		If $E$ is a Lebesgue measurable subset of $\T$ and $|E| \in (0, 2 \pi)$, then
		\begin{equation}\label{eq:Essential}
			\norm{ \HHH_{z^n \chi_E} }_e = \tfrac{1}{2}, \quad n \in \Z.
		\end{equation}
	\end{Lemma}

	\begin{proof}
		Since the Hankel matrix $\HHH_{z^n \chi_E}$ is obtained
		from $\HHH_{\chi_E}$ by either eliminating or inserting a finite number of columns, we conclude that 
		$$\norm{\HHH_{z^n \chi_E}}_e = \norm{\HHH_{\chi_E}}_e = \dist(\chi_E, H^\infty+\mathcal{C}) = \tfrac{1}{2}$$
		by \eqref{eq:EssentialDistance} and Lemma \ref{Lemma:HC}.
	\end{proof}
	\vskip .05in

	\begin{Remark}
	\hfill 
	\begin{enumerate}
	\item There is a weaker version  of Lemma \ref{Lemma:HC} in \cite[VII.A.2]{Koosis} namely
	$$\mbox{dist}(\psi, H^{\infty}) = 1,$$
	where $\psi = 1$ on a subset $E \subset \T$ (with $|E| \in (0, 2 \pi)$) and $\psi = -1$ on $\T \setminus E$. 
	\item If $E$ is a finite union of arcs, then another proof of Lemma \ref{Lemma:Essential} becomes available.
	Indeed, recall that if $\psi:\T\to\C$ is piecewise continuous, then \cite{MR628125} (see also \cite[Thm.~1.5.18]{Peller}) asserts that
	\begin{equation*}
		\norm{\HHH_{\psi}}_{e} = \tfrac{1}{2} \max_{\xi \in \T} |\psi(\xi^{+}) - \psi(\xi^{-})|,
	\end{equation*}
	which immediately yields \eqref{eq:Essential}.
	\end{enumerate}	
	\end{Remark}

\subsection{Truncated Toeplitz operators}
	In order to study $\Lambda_n(E)$ for $n \geq 1$, 
	we require a few facts about truncated Toeplitz operators, a class of operators whose study
	was spurred by a seminal paper of Sarason \cite{Sarason} (see \cite{RPTTO} for a current survey of the subject).
	Although much of the following can be phrased in terms of large truncated Toeplitz matrices \cite{Bottcher},
	the arguments involve reproducing kernels and conjugations which are more natural in the setting of truncated 
	Toeplitz operators \cite{NLEPHS, CFT,VPSBF}.  

	For $n \geq 1$, a simple computation with Fourier series shows that 
	$$(z^n H^2)^{\perp} := H^2 \ominus z^n H^2$$ is the finite dimensional vector space of  polynomials of degree at most $n - 1$. For $\psi \in L^{\infty}$ we consider the corresponding
	\emph{truncated Toeplitz operator}
$$A_{n, \psi}: (z^n H^2)^{\perp} \to (z^n H^2)^{\perp}, \quad A_{n, \psi} f = P_{n}(\psi f),$$
	where $P_{n}$ is the orthogonal projection of $L^2$ onto $(z^n H^2)^{\perp}$.   With respect to the  orthonormal basis $\{1, z, z^2, \cdots, z^{n- 1}\}$ for $(z^n H^2)^{\perp}$ the matrix representation of $A_{n, \psi}$ is the Toeplitz matrix $(\widehat{\psi}(j - k))_{0 \leq j, k \leq n  - 1}$.	
	
	It is easy to see that the map $C_n: (z^n H^2)^{\perp} \to (z^n H^2)^{\perp}$, defined in terms of boundary functions by
	\begin{equation*}
		(C_n f)(\zeta) = \overline{f(\zeta)} \zeta^{n - 1}, \quad \zeta \in \T, 
	\end{equation*}
	is a conjugate-linear, isometric, involution (i.e., a \emph{conjugation}).  
	Viewed as a mapping of functions on $\D$, the conjugation $C$ has the explicit form 
	$$C_n \sum_{j = 0}^{n - 1} a_{j} z^j = \sum_{j = 0}^{n - 1} \overline{a_{n - 1 - j}} z^{j}, \quad z \in \D.$$
	Moreover, it is also known that
	$A_{n, \psi}$ satisfies 
	$$A_{n, \psi}^{*} = C_n A_{n, \psi} C_n,$$ (i.e., $A_{n, \psi}$ is a \emph{complex symmetric operator} \cite{G-P, G-P-II}).
	Consequently,
	\begin{equation*}
		\|A_{n, \psi}\| = \max\big\{|\langle A_{n, \psi} f, C_n f\rangle|: \,f \in (z^n H^2)^{\perp}, \,\|f\|_{2} = 1 \big\}.
	\end{equation*}
	See \cite{NLEPHS} for a proof of the preceding result.

\section{Evaluation of $\Lambda_{n}(E)$ for $n \leq 0$}\label{Section:Negative}

	It turns out that we can compute  $\Lambda_{n}(E)$ exactly for $n \leq 0$.  In this setting,
	$\Lambda_{n}(E)$ is, to a large extent, independent of $n$ and the set $E$ itself.
	Much of the groundwork for this next result has already been done in 
	 Section \ref{Section:Preliminaries}.
	
	\begin{Theorem}\label{Theorem:Main}
		If $E$ is a Lebesgue measurable subset of $\T$ with $|E| \in (0, 2 \pi)$ then
		\begin{equation*}
			\Lambda_{n}(E) = \tfrac{1}{2}, \quad n \leq 0.
		\end{equation*}
	\end{Theorem}
	
	\begin{proof}
	
		By duality, note that for any $n \leq 0$ we have
		\begin{align*}
			\Lambda_{n}(E) & = \dist(z^{-n} \chi_E, H^{\infty})\\
			&  = \inf\{\|z^{-n} \chi_{E} - g\|_{\infty}: g \in H^{\infty}\}\\
			& \leq  \inf \{\|z^{- n} \chi_{E} - z^{- n} g\|_{\infty}: g \in H^{\infty}\}\\ 
			& =   \inf \{\|\chi_{E} - g\|_{\infty}: g \in H^{\infty}\}\\ 
			& \leq \tfrac{1}{2}.
		\end{align*}
		The last inequality follows since the constant function $g \equiv \frac{1}{2}$ belongs to $H^{\infty}$.
		To establish the reverse inequality, we observe that
		\begin{equation*}
			\Lambda_{n}(E) = \norm{ \HHH_{z^{-n} \chi_E}} \geq \norm{ \HHH_{z^{-n} \chi_E}}_e = \tfrac{1}{2}
		\end{equation*}
		by Lemma \ref{Lemma:Essential}.
	\end{proof}

	Although Theorem \ref{Theorem:Main} is quite definitive, its proof is not constructive.  It is therefore of
	interest to see if, whenever we are presented with a subset $E$ of $\T$ with $|E| \in (0, 2 \pi)$ and an $\epsilon > 0$, we can explicitly construct a function $F$ in the unit ball of $H^1$ for 
	which the quantity
	\begin{equation*}
		\left|  \frac{1}{2\pi i } \int_E F(\zeta)\zeta^n \,d\zeta  \right|
	\end{equation*}
	comes within $\epsilon$ of $\frac{1}{2}$.
	For general $n$ and $E$, this is most likely an extremely difficult problem.  However,
	in the special case where $n=0$ and $E$ is a finite union of arcs, the following method of S.~Ja.~Khavinson  \cite[p.~18]{Kha-Sr}
	furnishes a relatively explicit sequence of functions for which this occurs.

	Fix $N \geq 1$ and let $E$ be the disjoint union of $N$ open arcs of $\T$, the $j$th arc proceeding counterclockwise from
	$a_j = e^{i\alpha_j}$ to $b_j = e^{i \beta_j}$ so that $0 <\beta_j - \alpha_j < 2 \pi$ and 
	\begin{equation*}
		0 <  \sum_{j=1}^N (\beta_j - \alpha_j)  < 1.
	\end{equation*}
	Let $\chi$ denote the harmonic extension of $\chi_E:\T\to\C$
	to the open unit disk $\D$ and normalize the harmonic conjugate $\widetilde{ \chi }$ of $\chi$ by 
	requiring that $\widetilde{ \chi}(0) = 0$. Next observe that
	\begin{equation*}
		g := \exp\left[  \pi ( - \widetilde{\chi} + i \chi ) \right]
	\end{equation*}
	is an outer function which maps $\D$ onto the open upper half-plane $\mathbb{H}$
	and satisfies $\arg g = \pi \chi_E$ almost everywhere on $\T$.  The function
	\begin{equation*}
		\phi := \frac{g - i}{g + i}
	\end{equation*}
	is therefore inner and satisfies $\phi(E) = \mathbb{H} \cap \T$ \cite{MR2021044}.
	
	Now fix $\epsilon > 0$ and let $\Omega$ denote the rectangle with vertices
	$$\left( \pm \left( \tfrac{1 - 2\epsilon}{4 N} \right), \pm \left( \tfrac{ \epsilon}{2 N} \right) \right).$$
		Letting $\Gamma$ denote the boundary of $\Omega$, oriented in the positive sense,
	we note that the length of $\Gamma$ is $\ell(\Gamma) = \frac{1}{N}$.
	Let $\rho:\D\to\Omega$ be a conformal mapping such that $\mathbb{H} \cap \T$ is mapped
	to the portion of $\Gamma$ running, in the positive sense,
	from $( \frac{1-2\epsilon}{4 N},0)$ to $(- \frac{1-2\epsilon}{4 N},0)$.
	Now define $\psi := \rho \circ  \phi$ and note that $\psi$ maps $\D$ onto $\Omega$
	while sending each of the $N$ arcs of $E$ onto the upper half of $\Gamma$.
	In other words, we have 
	$$\psi(a_j) = \frac{1-2\epsilon}{4 N} \quad \mbox{and}\quad \psi(b_j) = -\frac{1-2\epsilon}{4 N}, \quad j=1,2,\ldots, N.$$
	
	Now define 
	$$F(z) := 2 \pi i \psi'(z)$$ and get 
	\begin{equation*}
		\norm{F}_1 =
		\frac{1}{2\pi} \int_{-\pi}^{\pi} |2 \pi i \psi'(e^{it})|\,dt
		= \int_{-\pi}^{\pi} |\psi'(e^{it})|\,dt = N \ell(\Gamma) = 1
	\end{equation*} along with
	\begin{align*}
		\tfrac{1}{2} - \epsilon
		&= N \left( \frac{1 - 2\epsilon}{2 N} \right) \\
		&= \sum_{i=1}^N [\psi(b_i) - \psi(a_i) ] \\
		&= \sum_{i=1}^N \int_{a_i}^{b_i} \psi'(\zeta)\,d\zeta\\
		&=  \frac{1}{2 \pi i} \int_{\T} \chi_E(\zeta) \underbrace{ (2 \pi i \psi'(\zeta) )}_{F(\zeta)}\,d\zeta.
	\end{align*}

\begin{Remark}
As noted in \cite[VII.A.2]{Koosis} there is a unique solution $g$ to the $H^{\infty}$ distance extremal problem $\mbox{dist}(\chi_{E}, H^{\infty})$. However, there is no maximizing $F$ for the extremal problem $\Lambda_{0}(E)$. Thus approximate solutions $F$ as in the above Khavinson construction is about the best one can do. 
\end{Remark}

\section{Estimating $\Lambda_n(E)$ for $n \geq 1$}\label{Section:Positive}
	Although we were able to explicitly evaluate $\Lambda_{n}(E)$
	for $n \leq 0$, the situation for $\Lambda_n(E)$ with $n \geq 1$ is substantially more
	complicated.  At this point, we are able to provide a variety of general estimates, along with
	explicit evaluations in a few very special cases (see Section \ref{se:interval}). We start with a simple, but relatively crude, estimate.

%

	\begin{Theorem}\label{Theorem:Estimate}
		If $E$ is a Lebesgue measurable subset of $\T$ with $|E| \in (0, 2 \pi)$ and $n \geq 1$, then 
		\begin{equation}\label{eq:Estimate}
			\max\big\{ \tfrac{1}{2}, \frac{|E|}{2 \pi} \big\} \leq \Lambda_n(E)<1.
		\end{equation}
		In particular, it follows that
		\begin{equation*}
			\lim_{|E|\to 2\pi} \Lambda_n(E) = 1
		\end{equation*}
		for each fixed $n \geq 1$.
	\end{Theorem}

	\begin{proof}
	From \eqref{increasinginn} we know that 
			$$\Lambda_n(E) \geq \Lambda_0(E) = \tfrac{1}{2}.$$
		 On the other hand, setting $F(z) = z^{n-1}$ in the definition
		\eqref{eq-Lambda} yields $$\Lambda_n(E) \geq \frac{|E|}{2 \pi},$$ which establishes the lower bound
		in \eqref{eq:Estimate}.
		
		To get the upper bound, we  first observe that $\Lambda_n(E) \leq 1$ by definition. Suppose, to get a contradiction,
		that $\Lambda_n(E)=1$.  Since 
		$\Lambda_n(E) = \norm{\HHH_{\bar z^n \chi_E}},$
		it follows from Lemma \ref{Lemma:Essential} that
		$$\|\HHH_{\bar z^n \chi_E}\|_e = \tfrac{1}{2} < 1 = \|\HHH_{\bar z^n \chi_E}\|.$$
		In light of the fact, from \eqref{norm-Hankel-dist}, that $$\norm{\bar z^n \chi_E}_{\infty} = 1 = \|\HHH_{\bar z^n \chi_E}\| = \dist(\bar z^n \chi_E,H^{\infty}),$$
		it follows from Lemma \ref{Lemma:Unimodular} (i) that
		$\bar z^n \chi_E$ has unit absolute value almost everywhere on $\T$, and this is obviously not true. This proves the
		upper bound in \eqref{eq:Estimate}.
	\end{proof}
	
	For $n  \geq 1$, the lower bound in \eqref{eq:Estimate} is somewhat crude.
	The following result is more precise and can be used to obtain numerical estimates of  $\Lambda_n(E)$.

	\begin{Theorem}\label{Theorem:Fejer}
		If $n \geq 1$, then for any $\alpha \in [-\pi,\pi]$ we have
		\begin{equation}\label{eq-Fejer}
			\Lambda_{n}(E) \geq \frac{1}{2\pi}\int_{-\pi}^{\pi} \chi_E(e^{it}) F_n(t-\alpha)\,dt,
		\end{equation}
		where $F_n$ denotes the Fej\'er kernel
		\begin{equation*}
			F_n(x) = \frac{\sin^2 (\frac{nx}{2})}{n\sin^2 (\frac{x}{2})}.
		\end{equation*}
		
	\end{Theorem}	

	\begin{proof}
		With $n \geq 1$ and $\alpha \in [-\pi, \pi]$ fixed, let $\xi = e^{i \alpha}$ and let 
		\begin{equation*}
			k_{\xi}(z) = \frac{1}{\sqrt{n}} \sum_{j = 0}^{n - 1} (\overline{\xi} z)^{j}
		\end{equation*}
		denote the corresponding normalized reproducing kernel for $(z^n H^2)^{\perp}$.
		Applying \cite[Thm.~3.1]{CFT} and \cite[Thm.~1]{NLEPHS} we obtain
		\begin{align}
		\Lambda_n(E)
			&=  \sup_{F \in b(H^1)} \left| \frac{1}{2\pi i } \int_{\T} F(\zeta)\chi_E(\zeta) \overline{\zeta}^{n}\,d\zeta \right| \nonumber \\
			&= \sup_{f \in b(H^2)} \left| \frac{1}{2\pi i } \int_{\T}  f(\zeta)^2 \chi_E(\zeta) \overline{\zeta}^{n}\,d\zeta \right| \nonumber \\
			&= \sup_{f \in b(H^2)} | \inner{ \chi_E  f, \overline{f} \zeta^{n-1}} | \nonumber \\
			&\geq \sup_{f \in b((z^n H^2)^{\perp})} | \inner{ \chi_E  f, \overline{f} \zeta^{n-1}} | \nonumber \\
			&= \sup_{f \in b((z^n H^2)^{\perp})} | \inner{ \chi_E  f, P_{n}(\overline{f} \zeta^{n-1})} | \nonumber \\
			&= \sup_{f \in b((z^n H^2)^{\perp})} | \inner{ A_{n, \chi_E} f, C_n f} | \nonumber \\
			&= \norm{A_{n, \chi_E}} \nonumber \\
			&\geq  | \inner{A_{n, \chi_E} k_{\xi}, k_{\xi}} |\nonumber \\
			&\geq \frac{1}{n| \xi |^{n-1}} \left| \frac{1}{2\pi} \int_{-\pi}^{\pi} \chi_E(e^{it})
				\left| \frac{ e^{int} - \xi^n }{ e^{it} - \xi} \right|^2\,dt\right| \nonumber \\
			&= \frac{1}{2 \pi} \int_{-\pi}^{\pi} \chi_E(e^{it})
				\frac{ \sin^2 \left( \frac{n}{2}(t - \alpha) \right) }{n \sin^2 \left(\frac{t-\alpha}{2}\right)}\,dt \nonumber \\
			&= \frac{1}{2\pi}\int_{-\pi}^{\pi} \chi_E(e^{it}) F_n(t-\alpha)\,dt, \label{eq-LLFSB}
		\end{align}
		where we have used the fact that $\chi_E$ and $F_n$ are both nonnegative.
	\end{proof}

	%

	\begin{Remark}
	Although not directly related to our investigations, it is worth noting that 
	the proof of Theorem \ref{Theorem:Fejer} can be used to obtain the  well-known fact that
		 the norm of the Toeplitz operator $T_{\psi}$ on $H^2$
		is given by $\norm{T_{\psi}}= \norm{\psi}_{\infty}$. 
		Indeed, computations similar to  \eqref{eq-LLFSB} lead to 
		\begin{equation*}
			\| A_{n, \psi} \| \geq  \lim_{n\to\infty} \left| \frac{1}{2\pi}\int_{-\pi}^{\pi} \psi(e^{it}) F_n(t-\alpha)\,dt \right| = | \psi(e^{i\alpha})|
		\end{equation*}
		 for almost every $\alpha$ in $[-\pi,\pi]$.
		Since $A_{n, \psi}$ is a compression of $T_{\psi}$ it follows that
		\begin{equation*}
			\norm{\psi}_{\infty} \geq \norm{ T_{\psi} } \geq \| A_{n, \psi} \|
			\geq  | \psi(\xi)|
		\end{equation*}
		for almost every $\xi$ in $\T$.  Therefore $\norm{T_{\psi}}= \norm{\psi}_{\infty}$, as claimed.
		\end{Remark}
	

It is easy to see that for any $\psi\in L^\infty$ we have $\lim_{n\to\infty}\dist(\psi, z^n H^\infty)=\|\psi \|_\infty$. Indeed, assume $\dist(\psi, z^n H^\infty)=\|\psi-z^ng_n\|_\infty$ (as noted in Section~\ref{Section:Preliminaries}, the distance is attained). If  
$\mathfrak{P}(\psi)$ is the Poisson extension of $\psi$ inside $\bbD$ (see, for instance,~\cite{Ga}), then
\[
\sup_{z\in\bbD} |\mathfrak{P}(\psi)(z)-z^n g_n(z)|\le\|\psi-z^n g_n\|_\infty.
\]
Since $\|g_n\|\le 2\|\psi\|$, we have $z^n g_n(z)\to0$ for any $z\in\bbD$, whence 
\[
\lim_{n\to\infty}\sup_{z\in\bbD} |\mathfrak{P}(\psi)(z)-z^n g_n(z)|\ge \sup_{z\in\bbD}|\mathfrak{P}(\psi)(z)|=\|\psi\|_\infty.
\]
In particular, it follows that $\lim_{n\to\infty} \Lambda_n(E) = 1$ whenever $E$ is a Lebesgue measurable subset of $\T$ with $|E| \in (0, 2 \pi]$.
  Under certain circumstances, we can obtain a better estimate for the speed of this convergence.
	
	\begin{Proposition}
		If $E$ is a Lebesgue measurable subset of $\T$ which contains a non-degenerate arc, then
		\begin{equation*}
			1 - \Lambda_n(E) = O(\tfrac{1}{n}), \quad n \to \infty.
		\end{equation*}
	\end{Proposition}
	
	\begin{proof}
		Without loss of generality, we may assume that $E$ contains the circular arc
		$I$ from $e^{-i\alpha}$ to $e^{i\alpha}$ where $\alpha \in (0, \pi)$.
		In light of Theorem \ref{Theorem:Fejer} we conclude that
		\begin{align*}
			\Lambda_n(E)
			&\geq \frac{1}{2\pi} \int_{-\alpha}^{\alpha}F_n(x)\,dx \\
			&= 1 - \frac{1}{2\pi} \int_{\alpha \leq |x| \leq \pi}F_n(x)\,dx \\
			&= 1 - \frac{1}{2\pi} \int_{\alpha \leq |x| \leq \pi} \frac{ \sin^2( \frac{nx}{2})}{n \sin^2(\frac{x}{2})}\,dx \\
			&\geq 1 - \frac{\alpha}{\pi n \sin^2 (\frac{\alpha}{2})}.\qedhere
		\end{align*}
	\end{proof}

	\begin{Question}
		Suppose that $E$ is a totally disconnected subset of $\T$ which has positive measure
		(for instance, if $E$ is a `fat Cantor set').  What can be said about the rate at which $\Lambda_n(E)$ tends to $1$?
	\end{Question}

\begin{Example}\label{ExampleRoss}
	We remark that we are free to maximize the lower bound 
	\eqref{eq-Fejer} with respect to the parameter $\alpha \in [-\pi, \pi]$.  If $t  \in (0, \pi)$ and $E_{t}$ denotes the arc of $\T$ from $e^{-i t}$ to $e^{i t}$ then evaluating the right hand side of \eqref{eq-Fejer} when $\alpha = 0$ gives us the integral 
	$$\frac{1}{2 \pi} \int_{-t}^{t} \frac{1}{n} \left(\frac{\sin(n x/2)}{\sin(x/2)}\right)^2 d x.$$ This integral can be computed directly yielding the following lower estimates
	\begin{align*}
		\Lambda_{0}(E_t) &= \tfrac{1}{2}, \\
		\Lambda_{1}(E_t) &\geq \frac{t}{2}, \\
		\Lambda_{2}(E_t) &\geq \frac{t}{2} + \frac{\sin (t)}{\pi},\\
		\Lambda_{3}(E_t) &\geq \frac{t}{2} + \frac{4 \sin (t)+\sin (2 t)}{3 \pi }, \\
		\Lambda_{4}(E_t) &\geq \frac{t}{2} + \frac{3 \sin (t)+\sin (2 t)+\frac{1}{3} \sin (3 t)}{2 \pi }.
	\end{align*}
From the estimate 
$$\max\big\{ \tfrac{1}{2}, \frac{|E_t|}{2 \pi} \big\} \leq \Lambda_n(E_t)$$ in \eqref{eq:Estimate} we observe that the above lower estimates only become meaningful when the right hand sides of the above expressions are greater then $\frac{1}{2}$ (which will happen when $t$ is bounded away from $0$). 
As noted in \eqref{rotation-invariant}, the above estimates hold for any arc of $\T$ with length $2 t$.  
\end{Example}

 \section{The case of the arc}\label{se:interval}

Suppose that $\alpha \in (0, \pi)$ and let $I_{\alpha}=\{e^{it}:t\in(-\alpha, \alpha)\}$. 
We assume $n \geq 1$. 
It is known \cite[p.~146]{Koosis} that the infimum in \eqref{eq:definition of lambda_n with distance} is attained. 
We will compute the minimizing function, thus obtaining a formula for $\Lambda_n(I_{\alpha})$ which is explicit when $n = 1$. As noted in \eqref{rotation-invariant}, our formula will hold not only for $I_{\alpha}$ but for any arc of $\T$ with length $2 \alpha$. Moreover, we will see in Section~\ref{se:more general sets} that for $n=1$ it can be extended to any measurable set.

%
%
  
\subsection{Conformal maps} 
We will use the notation 
$$\wideparen{(\zeta, \eta)}$$
for
$\zeta\not= \eta\in\T$ to denote the sub-arc of $\T$  \emph{from} $\zeta$ \emph{to} $\eta$ in the positive direction.
\begin{Remark}
To avoid confusion later on, it is important to take careful note of the direction one traverses the arc $\wideparen{(\zeta,  \eta)}$. One needs to traverse this arc \emph{from} $\zeta$ \emph{to} $\eta$ always keeping $\D$ on the left. For example,
$$\wideparen{(e^{i \pi/4}, e^{-i\pi/4})}$$ is the arc which travels the \emph{long way} around the circle from $e^{i \pi/4}$ to $e^{- i \pi/4}$ while 
$$\wideparen{(e^{-i \pi/4}, e^{i\pi/4})}$$ travels the short way around. 
\end{Remark}

For fixed $n \geq 1$ and $r \in [\frac{1}{2}, 1]$ let $\Onr$ be the domain in $\C$ defined by
\begin{equation*}
	\Onr := 
	\D \setminus \left\{ \Big|z^n-\frac{1}{r} \Big| \leq1, -\frac{\pi}{2n} \leq \arg z \leq \frac{\pi}{2n} \right\}.
\end{equation*}
This domain $\Onr$ is obtained as follows. The pre-image of the closed unit disk centered at $\frac{1}{r}$ via the mapping $z\mapsto  z^n$ has $n$ components (see Figure \ref{FigureX1}).
We form $\Onr$ by removing from $\D$ the component containing $1$  (see Figure \ref{FigureX2}). Then $\Onr$ is a simply connected domain which is symmetric with respect to $\R$. For $n \ge 1$ note that $\OOO_{n,1/2}=\D$. Also note that 
\begin{equation} \label{OOO-dec}
\OOO_{n, r'} \subset \OOO_{n, r}, \quad r' > r.
\end{equation}
\begin{figure}
\centering
\begin{subfigure}{.5\textwidth}
  \centering
  \includegraphics[width=.9\linewidth]{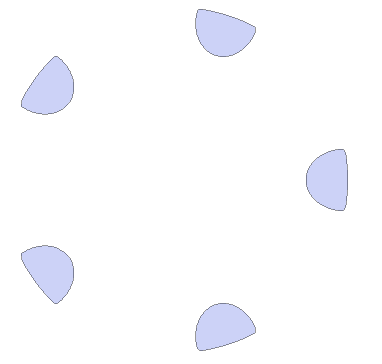}
  \caption{}
  \label{FigureX1}
\end{subfigure}%
\begin{subfigure}{.5\textwidth}
  \centering
  \includegraphics[width=.9\linewidth]{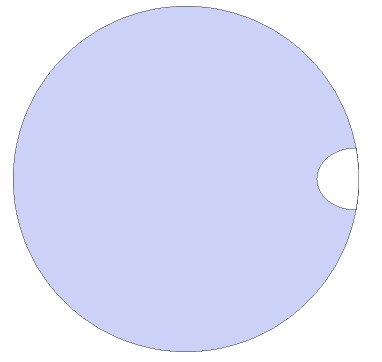}
  \caption{}
  \label{FigureX2}
\end{subfigure}
\caption{{\footnotesize (A) The components (shaded) of the pre-image of the closed unit disk centered at $4/5$ of the mapping $z \mapsto z^5$. (B) The region (shaded) $\OOO_{5, 4/5}$.}}
\end{figure}
	We let $w^\pm_{n,r}$ be the two `corners' of $\partial\Onr$ characterized by 
$$\partial\Onr\cap \T= \wideparen{(w^+_{n,r}, w^-_{n,r})}.$$ 

\begin{Lemma}\label{le:conformal general n} Suppose $\alpha \in (0, \pi)$. For every $n \geq 1$
there exist an $r_{n,\alpha} \in (\frac{1}{2}, 1)$ and a conformal homeomorphism $$\Phi_{n,\alpha}:\bbD\to\OOO_{n, r_{n,\alpha}}$$
 such that $\Phi_{n,\alpha}(0)=0$ and, denoting $\Phi_{n,\alpha}$ to also be its continuous extension to $\D^{-}$,  we have $\Phi_{n,\alpha}(e^{\pm i\alpha})=w^\pm_{n,r_{n, \alpha}}$.
\end{Lemma}

\begin{proof}
Fix $n \geq 1$ and $r \in (\frac{1}{2}, 1)$. The domain $\Onr$ is simply connected and so there is a unique conformal homeomorphism $\phi_r$ satisfying
$$\phi_r:\bbD\to\Onr, \quad \phi_r(0)=0, \quad \phi_r'(0)>0.$$ Since $\Onr$ is symmetric with respect to $\R$, it is easy to see that $\overline{\phi_r(\bar z)}$ satisfies the same conditions, and thus by uniqueness, we have $\phi_r(z)=\overline{\phi_r(\bar z)}$. In particular, $\phi_r((-1, 1))\subset \bbR$. We also see that $\phi_r$ extends continuously to $\T$ and satisfies the conditions
$$\phi_r(-1)=-1 \quad \mbox{and} \quad \phi_r(1)=\left(\frac{1-r}{r}\right)^{1/n}.$$ Finally, again by uniqueness, $\phi_{1/2}(z) = z$, i.e., $\phi_{1/2}$ is the identity map.

If $r_k, r\in [\frac{1}{2}, 1)$ and $r_k\to r$, one sees first that the domains $\OOO_{n,r_k}$ and $\Onr$ satisfy the hypothesis of the Carath\'{e}odory kernel theorem \cite[Theorem 1.8]{Po} and thus $\phi_{r_k}(z)\to\phi_r(z)$ for all $z\in \bbD$. Then  $\OOO_{n,r_k}$ and $\phi_{r_k}$ satisfy the hypotheses of~\cite[Corollary 2.4]{Po}, whence it follows that $\phi_{r_k}\to\phi_r$ uniformly on $\T$. See Figure \ref{FigureX3} for an illustration of the dependence of the domain $\OOO_{n, r}$ on the parameter $r$.

\begin{figure}
\centering
\begin{subfigure}{.335\textwidth}
  \centering
  \includegraphics[width=.9\linewidth]{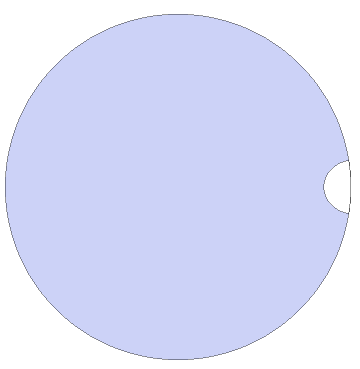}
  \caption{}
  \label{}
\end{subfigure}%
\begin{subfigure}{.335\textwidth}
  \centering
  \includegraphics[width=0.9\linewidth]{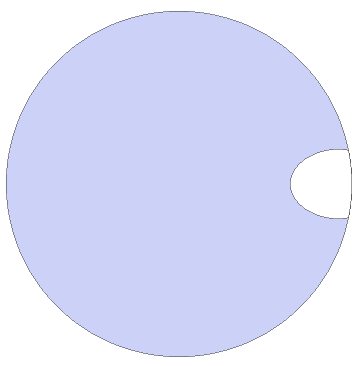}
  \caption{}
  \label{}
\end{subfigure}%
\begin{subfigure}{.335\textwidth}
  \centering
  \includegraphics[width=0.9\linewidth]{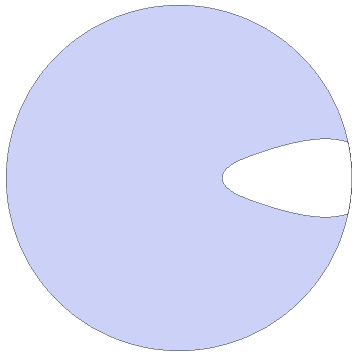}
  \caption{}
  \label{}
\end{subfigure}
\caption{{\footnotesize (A) $\OOO_{5, .07}$, (B) $\OOO_{5, 0.9}$, (C) $\OOO_{5, 0.999}$}}
\label{FigureX3}
\end{figure}

In particular, for our fixed $\alpha \in (0, \pi)$, the map from $[\frac{1}{2}, 1)$ to $\C$ defined by  $r\mapsto \phi_r(e^{i\alpha})$ is continuous.

Recalling that 
$\wideparen{(e^{i \alpha}, e^{-i\alpha})}$ goes from $e^{i \alpha}$ to $e^{-i\alpha}$ the\emph{ long way} around $\T$ and, similarly, 
$\wideparen{(w^+_{n,r}, w^{-}_{n,r})}$ goes from the upper corner $w^{+}_{n, r}$ to the lower corner $w^{-}_{n, r}$ the\emph{ long way} around $\T$, 
suppose that  $r\in [\frac{1}{2}, 1)$ is such that 
\begin{equation}\label{eq:arcs}
 \phi_r(\wideparen{(e^{i\alpha}, e^{-i\alpha})})\subset \wideparen{(w^+_{n,r}, w^-_{n,r})}.
\end{equation}
Then $\phi_r$ can be continued by Schwarz reflection  to a function $\Phi_r$ analytic in
 $$\widehat{\C} \setminus \wideparen{(w^-_{n,r}, w^+_{n,r})},$$ and the range of any such $\Phi_r$ (for fixed $r_0 > \frac{1}{2}$) does not contain a fixed neighborhood of the point 1. 

Suppose now that~\eqref{eq:arcs} is true for \emph{every} $r\in [\frac{1}{2}, 1)$. For a sequence $r_k\to 1$, the functions $$\frac{1}{\Phi_{r_k}-1}, \quad k \geq 1,$$ form a normal family in the domain
$$\Omega_n:= \widehat{\C}\setminus \wideparen{(e^{-i\pi/3n}, e^{i\pi/3n})},$$
the intersection of the decreasing (see \eqref{OOO-dec}) domains 
$$\widehat{\C} \setminus \wideparen{(w^{-}_{n, r_{n_k}}, w^{+}_{n, r_{n, k}})}, \quad  k \geq 1.$$

By passing to a subsequence, we may assume that $(\Phi_{r_k}-1)^{-1}$ converges uniformly on compact subsets of $\Omega_n $. Thus $\Phi_{r_k}$ converges uniformly on compact subsets of $\Omega_n$ to some analytic function $g$. Since $\Phi_r(0)=0$ and $\phi_r(-1)=-1$ for all $r$, we must have $g(0)=0$ and $g(-1)=-1$. Thus $g$ is a non-constant analytic function  and so $g$ must be open. On the other hand, if 
$$\widetilde{\OOO}_{n, r} := \OOO_{n, r} \cup 
\left\{z \in \C: \frac{1}{z} \in \OOO_{n, r}
\right\}^{-},$$
then the image of $\Phi_{r_k}$ is contained in $\widetilde{\OOO}_{n, r_k}$. 
We will now derive a contradiction and show that $g$ is not an open map. Indeed, the image of $g$ is contained in 
$$\bigcap_{k} \widetilde{\OOO}_{n, r_k} = \widetilde{\OOO}_{n, 1}$$
and this last set contains $0$ in its boundary. But since $g(0) = 0$, we see that $g$ cannot be an open map. 

It now follows that~\eqref{eq:arcs} cannot be true for \emph{every} $r\in [\frac{1}{2}, 1)$. For $r=\tfrac{1}{2}$ we see that $\phi_{1/2}(e^{i\alpha})=e^{i\alpha}$ and $w^\pm_{n,1/2}=1$, so ~\eqref{eq:arcs} is satisfied. Clearly $w^\pm_{n,r}$ depends continuously on the parameter $r$. If we define
\[
r_{n,\alpha}:=\sup\left\{r\in[\tfrac{1}{2}, 1): \text{\eqref{eq:arcs} is true for any $s\in [\tfrac{1}{2}, r)$}\right\},
\]
then 
$$\phi_{r_{n,\alpha}}(e^{\pm i\alpha})=w^\pm_{n,r}.$$ Indeed, by taking a sequence $r_k\nearrow r_{n,\alpha}$ one sees that $|\phi_{r_{n,\alpha}}(e^{\pm i\alpha})|=1$. If, say, $\Arg(\phi_{r_{n,\alpha}}(e^{ i\alpha}))>\Arg w^+_{n,r}$, then, by continuity, this would happen for all $r>r_{n,\alpha}$ in a small neighborhood of $r_{n,\alpha}$, which is easily seen to contradict the definition of $r_{n,\alpha}$. 
It follows that $r_{n,\alpha}$ and 
 $\Phi_{n,\alpha}=\phi_{r_{n,\alpha}}$ satisfy the requirements of the lemma.
\end{proof}


\begin{Remark}\label{re:uniqueness r_n}
The uniqueness of $r_{n,\alpha}$ and $\Phi_{n,\alpha}$ subject to the conditions in Lemma~\ref{le:conformal general n} is a consequence of Theorem~\ref{th:lambda_n(interval)} (see below), since it is shown in its proof that $r_{n,\alpha}=\dist(\bar z^n\chi_{I_{\alpha}}, H^\infty)$, and that $\Phi_{n,\alpha}$ is uniquely defined by prescribing values at the three points $0, e^{-i\alpha}$ and $e^{i \alpha}$. 
\end{Remark}

\subsection{The heart of the matter} We have now arrived at the main part of our argument which requires some technical details of Hankel and Toeplitz operators.  Here is our main result. 

\begin{Theorem}\label{th:lambda_n(interval)}
If $\alpha \in [0, \pi]$, $I_{\alpha} = \wideparen{(e^{-i\alpha}, e^{i\alpha})}$, and  $n \geq 1$, then 
$$\Lambda_n(I_{\alpha})=r_{n,\alpha}.$$
\end{Theorem}

\begin{proof}
Fix $\alpha$ and $n$ and let $I = I_{\alpha}$ and define 
$$\phi:=r_{n,\alpha}\Phi_{n,\alpha}^n.$$
 From Lemma~\ref{le:conformal general n} it follows that $\phi$ is analytic on $\D$, continuous on $\D^{-}$, and satisfies 
 $$\phi(\wideparen{(e^{i\alpha}, e^{-i\alpha})}) \subset r_{n,\alpha}\T \quad \mbox{and} \quad \phi(\wideparen{(e^{-i\alpha}, e^{i\alpha})}) \subset 1+r_{n,\alpha}\T$$
\begin{figure}
		\begin{center}
			\includegraphics[width=0.4\textwidth]{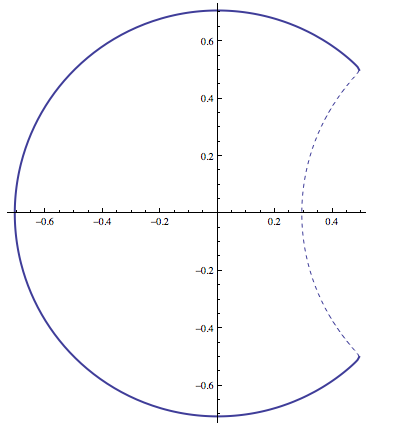}
		\end{center}
		\caption{\footnotesize If $n = 1$ then $\phi(\wideparen{(e^{i\alpha},e^{-i\alpha})}) \subset r_{1,\alpha}\T$ (solid) and $\phi(\wideparen{(e^{-i\alpha}, e^{i\alpha})}) \subset 1+r_{1,\alpha}\T$ (dashed). }
						\label{FigureGm12}
	\end{figure}
(see Figure \ref{FigureGm12}). 
   Therefore, the function $u: \T \to \C$ defined by
  $$u:=\bar z^n(\phi - \chi_{I})$$ has constant absolute value equal to $r_{n,\alpha}$. Also observe that $\phi$ has a zero of order $n$ at the origin allowing us to write $\phi=z^n\phi_1$ and thus 
  $$u=\phi_1 - \bar z^n\chi_{I}.$$

We want to apply Lemma \ref{Lemma:Unimodular} to $u$.  First, $u$ has constant absolute value on $\T$. Second 
\begin{equation}\label{Toe-id}
T_u=T_{\bar z^n}T_{\phi-\chi_I}
\end{equation} 
and so $T_u$ is Fredholm if and only if $T_{\phi-\chi_I}$ is Fredholm. 
Third, since $\phi$ is continuous we can use ~\eqref{eq:EssentialDistance} and Lemma \ref{Lemma:Essential} to see that  
$$\|\HHH_{\phi-\chi_I}\|_e= \|\HHH_{\overline{\phi} - \chi_{I}}\|_{e}  = \|\HHH_{\chi_{I}}\|_e = \tfrac{1}{2} <r_{n,\alpha}.$$ It now follows from Lemma \ref{Peller-2} that $T_{\phi - \chi_{I}}$ is Fredholm. 

To compute the index of $T_u$, note that $\phi - \chi_{I}$ is piecewise continuous, with two discontinuity points at $e^{-i\alpha}$ and $e^{i\alpha}$. By~\cite[Theorems 1 and 2]{Sa} we know that (i)  the harmonic extension $\mathfrak{P}(\phi - \chi_{I})(r, t)$ of $\phi-\chi_{I}$ is bounded away from zero in some annulus $\{z: 1-\epsilon<|z|<1\}$; (ii) for any fixed $r \in (1 - \epsilon, 1)$ the curve $t\mapsto \mathfrak{P}(\phi - \chi_{I})(r, t)$ has the same winding number with respect to $0$; (iii) the  index of $T_{\phi - \chi_{I}}$ is equal to minus this winding number. 

The compute this winding number, notice that the circles of radius $r_{n,\alpha}$ centered at $0$ and $1$ intersect at the two points $r_{n,\alpha}e^{\pm i\beta}$ where 
$\cos\beta=\frac{1}{2r_{n,\alpha}}$. We denote by $\gamma$ the curve obtained by considering the curve $(\phi-\chi_I)(\T)$ and then making it into a closed curve by adding, in the appropriate places, the segments 
$$[r_{n,\alpha} e^{i\beta}-1, r_{n,\alpha} e^{i\beta}]  \quad \mbox{and} \quad [r_{n,\alpha} e^{-i\beta}, r_{n,\alpha}e^{-i\beta}-1] $$ (see Figure \ref{winding12}).
\begin{figure}
\centering
\begin{subfigure}{.52\textwidth}
  \centering
  \includegraphics[width=1.0\linewidth]{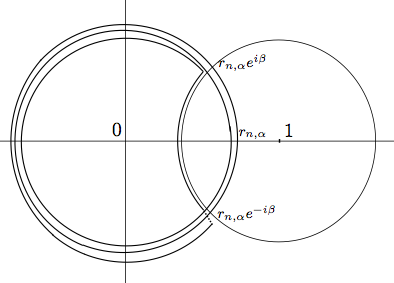}
  \caption{}
  \label{winding1}
\end{subfigure}%
\begin{subfigure}{.52\textwidth}
  \centering
  \includegraphics[width=1.0\linewidth]{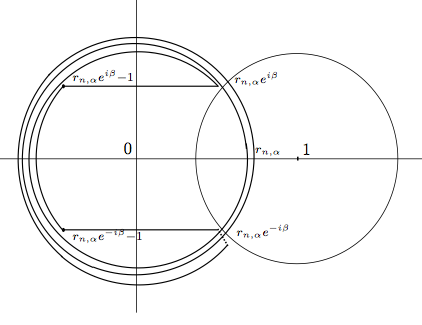}
  \caption{}
  \label{winding2}
\end{subfigure}
\caption{{\footnotesize (A) The curve $\phi(\T)$ when $n=3$. (B) The curve $\gamma$ corresponding to Figure \ref{winding1}. It has winding number $3 - 1 = 2$ with respect to the origin.}}
\label{winding12}
\end{figure}

This last curve has winding number $n-1$ with respect to the origin and so by \eqref{Toe-id}
\[
\ind T_u=\ind T_{\bar z^{n}}+\ind T_{\phi - \chi_{I}}  = n + (1 - n) =1.
\]

To finish, Lemma \ref{Lemma:Unimodular} (ii) tells us that 
\[
\begin{split}
 r_{n,\alpha}&=\|u\|_\infty=\dist(u, H^\infty)=\dist(\bar z^n\chi_{I}-\phi_1, H^\infty)\\&=\dist(\bar z^n\chi_{I}, H^\infty)=
\Lambda_n(I_{\alpha})
\end{split}
\]
which proves the theorem.
\end{proof}

\begin{Remark}
Since 
\[\|\HHH_{\bar z^n \chi_{I_{\alpha}}} \| =r_{n,\alpha}>1/2= \|\HHH_{\bar z^n \chi_{I_{\alpha}}} \|_e,
\]
it follows from a classical result of Adamyan--Arov--Krein (see~\cite[Theorem 1.1.4]{Peller}) that $g=\phi_1$
is the unique minimizing function in \eqref{eq:definition of lambda_n with distance}.

\end{Remark}

\section{An explicit computation for $n=1$}

When  $n=1$ one can make explicit computations. In this case 
$$\OOO_{1,r}= \D \setminus \left \{z: \left|z- \frac{1}{r}\right| \le1\right\}$$ 
\begin{figure}
		\begin{center}
			\includegraphics[width=0.4\textwidth]{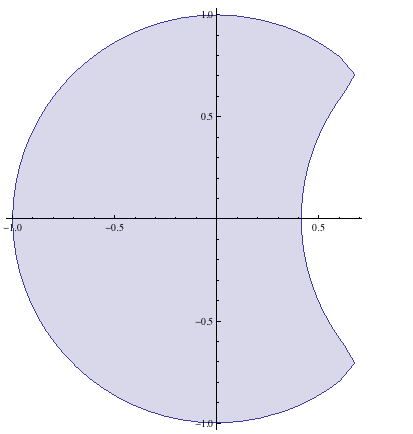}
		\end{center}
		\caption{\footnotesize The domain $\OOO_{1, r_{1, \pi/2}}$.}
				\label{FigureGm1}
	\end{figure}
(see Figure \ref{FigureGm1}) and the corresponding conformal homeomorphisms can be computed in closed form, leading to a precise formula for $\Lambda_1(I_{\alpha})$, where 
$$I_{\alpha} = \wideparen{(e^{-i\alpha}, e^{i\alpha})}.$$

\begin{Theorem}\label{th:formula for lambda_1}
For $\alpha \in (0, \pi)$ we have 
\begin{equation*}
 \begin{split}
  r_{1,\alpha}&=\frac{1}{2}\sec\left(\frac{\pi\alpha}{\pi+2\alpha}\right),\\  
  \Phi_{1,\alpha}(z)& =
  \frac{\left( \frac{z-e^{i\alpha}}{e^{i\alpha}z-1} \right)^\frac{\pi}{\pi+2\alpha}
-e^\frac{i\pi\alpha}{\pi+2\alpha}  }{\left( \frac{e^{i\alpha}z-e^{2i\alpha}}{e^{i\alpha}z-1} \right)^\frac{\pi}{\pi+2\alpha}-1 },
   \end{split}
\end{equation*}
where we have taken the principal branch of the power $\frac{\pi}{\pi+2\alpha}$.
Therefore, 
\begin{equation}
\label{eq:lambda for n=1}
\Lambda_1(I_{\alpha})=\frac{1}{2} \sec\frac{\pi\alpha}{\pi+2\alpha}.
\end{equation}
\end{Theorem}
 
\begin{proof}
Define $$\beta : =\frac{\pi\alpha}{\pi+2\alpha}.$$
It is easily checked, using the notation in Section~\ref{se:interval}, that $$e^{\pm i\beta}=w^\pm_{1,r_{1,\alpha}}.$$ If
\[
\phi(z)=\frac{z-e^{i\alpha}}{e^{i\alpha}z-1},  \quad
w(z)=\frac{z^{\beta/\alpha}-e^{i\beta}}{e^{i\beta}z^{\beta/\alpha}-1},
\]
then $\phi$ is a conformal homeomorphism from $\bbD$ to the upper half plane $\bbH$, while (see, for instance,~\cite[page 48]{Kob}), $w$ is a conformal homeomorphism from $\bbH$ to $\OOO_{1,r_{1,\alpha}}$.
 We have $\Phi_{1,\alpha}=w\circ \phi$, and one can check directly that $\Phi_{1,\alpha}(0)=0$ and $\Phi_{1,\alpha}(e^{\pm i\alpha})=e^{\pm i\beta}$. Thus $\Phi_{1,\alpha}$ satisfies the conditions of Lemma~\ref{le:conformal general n}.
\end{proof}
 
For instance, 
$$\Lambda_1(I_{\pi/2}) = \tfrac{1}{\sqrt{2}}.$$ It also follows from~\eqref{eq:lambda for n=1} that 
$\alpha \mapsto \Lambda_{1}(I_{\alpha})$ is an increasing function. 

\begin{Question}
Is $\alpha \mapsto \Lambda_{n}(I_{\alpha})$ an increasing function for every $n$? 
\end{Question}


\begin{Question}
When $E$ is an arc and $n=1$, is the supremum in \eqref{eq-Lambda} attained?
\end{Question}

\begin{Remark}
 When $n > 1$, it does not seem possible to obtain explicit formulas for $\Phi_{n, \alpha}$ and $r_{n, \alpha}$.
\end{Remark}
 
\section{More general sets}\label{se:more general sets}

The following lemma was proved by Nordgren~\cite{Nordgren} (see also~\cite[Theorem 7.4.1, Remark 9.4.6]{CMR}).

\begin{Lemma}\label{le:theta(0)=0}
 If $\theta$ is an inner function with $\theta(0)=0$, then $\theta$ is measure preserving as a transformation from $\T$ to itself.
\end{Lemma}

The next result appears in~\cite[Appendix]{Qiu}. We include the proof for completeness.

\begin{Theorem}\label{th:interpolation}
If $E\subset\T$ is a measurable set and $I\subset\T$ is an arc with $|I|=|E|$, then there exists an inner function $\theta$, with $\theta(0)=0$, such that  $\theta^{-1}(I)$ and $E$ are equal almost everywhere.
\end{Theorem}

\begin{proof} We may assume $0<|E|<2\pi$.
Let $v$ be the harmonic extension to $\D$ of $\chi_E$ and by $\tilde v$ its harmonic conjugate. Define $\sS:=\{z\in\C:0<\Re z<1\}$, $\delta_0:=\{z\in\C:\Re z=0\}$, $\delta_1:=\{z\in\C:\Re z=1\}$; then $\psi:=v+i\tilde v$ is an analytic map from $\D$ to $\sS$, such that the nontangential limit
$\psi(e^{it})$ is almost everywhere in $\delta_1$ for $e^{it}\in E$ and in $\delta_0$ for $e^{it}\in \T\setminus E$.

If $\tau:\sS\to \D$ is the Riemann  map that satisfies $\tau(\psi(0)=0$, then $I_0:=\tau(\delta_0)$ and $I_1:=\tau(\delta_1)$ are complementary arcs on $\T$, while $\phi:=\tau\circ \psi$ is an inner function that satisfies $\phi(0)=0$. Moreover, we have (up to sets of measure 0) $E\subset\phi^{-1}(I_1)$ and $\T\setminus E\subset\phi^{-1}(I_0) $. 

Apply Lemma~\ref{le:theta(0)=0}  to the inner function $\phi$ to see that $|\phi^{-1}(I_0)|=|I_0|$, $|\phi^{-1}(I_1)|=|I_1|$, whence $|\phi^{-1}(I_0)\cup\phi^{-1}(I_1)|=2\pi$. It follows then that (up to sets of measure 0) $E=\phi^{-1}(I_1)$ and $\T\setminus E=\phi^{-1}(I_0) $. We also have $|I|=|E|=|I_1|$ and therefore the required inner function $\theta$ can be obtained by composing $\phi$ with a rotation that maps $I_1$ onto $I$.
\end{proof}

When $|\partial E|=0$, an explicit formula for $\theta$ may be obtained from~\cite[Proposition 2.1]{BT}.

\begin{Theorem}\label{th:general sets}
Suppose $E\subset\T$, $|E| \in (0, 2 \pi)$. 
Then for any $n \geq 1$ we have
\[
\Lambda_n(E)\le \Lambda_n( \wideparen{(e^{-i |E|/2}, e^{i |E|/2})}).
\]
Moreover, 
$$\Lambda_1(E)= \Lambda_1(\wideparen{(e^{-i |E|/2}, e^{i |E|/2})})=r_{1,|E|/2}.$$
\end{Theorem}

Recall that the definition of $r_{n, \alpha}$ is given in Theorem \ref{th:formula for lambda_1}.

\begin{proof} Let 
$$I :=\wideparen{(e^{-i |E|/2}, e^{i |E|/2})}$$ and $\theta$ be corresponding inner function produced by Theorem~\ref{th:interpolation}.
Then $\theta$ is measure preserving and $\chi_{I}\circ\theta=\chi_{E}$. If $g\in z^nH^\infty$, then $g\circ \theta\in z^n H^\infty$, and obviously
\[
\|\chi_{I}-g\|_\infty=\|\chi_{I}\circ \theta-g\circ\theta\|_\infty=
\|\chi_{E}-g\circ\theta\|_\infty.
\]
By taking the infimum with respect to all $g\in z^nH^\infty$, we obtain 
$$\dist(\chi_{E}, z^nH^\infty)\le \dist(\chi_{I}, z^nH^\infty),$$ or $\Lambda_n(E)\le \Lambda_n(I)$.

To prove the opposite inequality in the case of $\Lambda_1$, note, from the fact that $\theta$ is measure preserving, that, if $F$ is in the unit ball of $H^1$, then $F\circ \theta$ is also in the unit ball of $H^1$. Moreover,
\[
\frac{1}{2\pi} \int_{-\pi}^{\pi} \chi_{I} F dt=
\frac{1}{2\pi} \int_{-\pi}^{\pi} (\chi_{I}\circ \theta)( F \circ\theta) dt=   \frac{1}{2\pi} \int_{-\pi}^{\pi} \chi_{E} (F\circ\theta )dt.
\]
By taking the supremum of the absolute value with respect to all $F$ in the unit ball of $H^1$, we obtain $\Lambda_1(I)\le \Lambda_1(E)$, which is the desired inequality.
\end{proof}
\vskip .05in

\noindent {\bf Acknowledgement:} The authors thank Damien Gayet for some useful suggestions concerning conformal mappings, and Gilles Pisier for bringing to our attention reference~\cite{Qiu}.

\bibliography{FourierExtremal} 

\end{document}